\documentclass{amsart} 
\usepackage{amsfonts}
\usepackage{amssymb} 

\usepackage{graphicx,color}         
\usepackage{amsthm}
\usepackage{hyperref}

\hypersetup{colorlinks,allcolors=black}

\newtheorem{theorem}{Theorem}
\newtheorem*{theorem*}{Theorem}
\newtheorem*{proof*}{}

\newtheorem{corollary}[theorem]{Corollary}

\newtheorem{lemma}[theorem]{Lemma}

\numberwithin{equation}{section}

\begin{document}
\title[area estimates for stable $H$-hypersurfaces]{A note on area estimates for stable $H$-hypersurfaces and applications}
\author{Marcos Ranieri}
\email{marcos.ranieri@im.ufal.br}
\address{Universidade Federal de Alagoas, Instituto de Matemática, 57072-970, Maceió - AL, Brazil}
\author{Elaine Sampaio}
\email{elaine.carlos@im.ufal.br}
\address{Universidade Estadual Vale do Acaraú, 62010-305, Sobral - CE, Brazil}
\curraddr{Universidade Federal de Alagoas, Instituto de Matemática, 57072-970, Maceió - AL, Brazil}
\author{Feliciano Vitório}
\email{feliciano@pos.mat.ufal.br}
\address{Universidade Federal de Alagoas, Instituto de Matemática, 57072-970, Maceió - AL, Brazil}

\begin{abstract}
We show some area estimates for stable CMC hypersurfaces immersed in Riemannian manifolds with scalar and sectional curvature bounded from below. In particular, we focus on immersions in three-dimensional Riemannian manifolds. As an application, we derive some upper estimates for the bottom spectrum of these hypersurfaces. This paper generalizes and builds on the techniques developed by Munteanu, Sung and Wang \cite{Mun} for stable minimal surfaces.
\end{abstract}

\maketitle

\section{Introduction}

A minimal hypersurface $\Sigma$ is said to be stable if it minimizes the area functional $\mathcal{A}(t)$ up to the second order for all compactly supported variations. A natural generalization of this concept is the stability of surfaces with constant mean curvature, which was introduced by Barbosa and Do Carmo \cite{LM} and has been extensively investigated in the last few decades, see for instance \cite{barbosa1988stability, MPR, mori1983stable} and references therein.

Surfaces with constant mean curvature $H$, also known as $H$-surfaces, correspond to critical points of the functional $$\mathcal{J}(t) = \mathcal{A}(t) - 2H\mathcal{V}(t),$$
where $\mathcal{A}(t)$ and $\mathcal{V}(t)$ denote the area and the volume enclosed of the variation between $0$ and $t$.
In this case, an $H$-hypersurface is said to be \textit{strongly stable} if, for compactly supported variations, we have 
$$\displaystyle \frac{d^2}{dt^2}\biggr\rvert_{t=0}\mathcal{J}(t) \geq 0.$$ It is well known that strong stability is equivalent to the positiveness of the first eigenvalue of the Jacobi operator $$L=\Delta +|h|^2+\mathrm{Ric(\nu , \nu )},$$ as well as the existence of a positive
smooth solution $u$ to the equation $Lu = 0$ (see \cite{FC}).  
Here, $h$ is the second fundamental form of $\Sigma$, $\mathrm{Ric}$ is the Ricci curvature of $M$ and $\nu$ is the outward unit normal vector field.  It is important to note that if $\Sigma$ is a strongly stable $H$-surface in $M$, then
 \begin{equation}
    \label{estabilityint}
    \int _{\Sigma} (|h|^2+\mathrm{Ric(\nu ,\nu )})f^2 \leq \int _{\Sigma} |\nabla f|^2, \ \ \ \forall f \in C_{0}^{\infty}(\Sigma).
\end{equation}
 For $H$-surfaces there exists another notion of stability related to the isoperimetric problem: a $CMC$ surface is said to be \textit{weakly stable} if \eqref{estabilityint} holds for every test function $f \in C_{0}^{\infty}(\Sigma)$ with $\int _{\Sigma} f=0$. A strongly stable $H$-surface is weakly stable. However, the converse may not be true. We emphasize that all of our results pertain to strongly stable $H$-surfaces, which we will refer to as stable $H$-surfaces for brevity. 

Recently, Munteanu, Sung and Wang \cite{Mun} were able to show some nice area estimates for complete stable minimal surfaces in a three-dimensional manifold with bounded scalar curvature or sectional curvature. Inspired by \cite{Mun}, we show some area estimates for stable $H$-surfaces. Our first result can be stated as

\begin{theorem}
\label{int1} Let $B_{p}\left( R\right) $ be a geodesic ball in a stable
$H$-surface $\Sigma $, with $|H|<1$, in a three-dimensional manifold $M.$ Assume that $%
B_{p}\left( R\right) $ does not intersect the boundary of $\Sigma.$ 

\begin{enumerate}
\item If the scalar curvature $S$ of $M$ satisfies $S\geq -6,$ then

\begin{equation*}
A\left( R\right) \leq C_{1}\,e^{2\sqrt{1-H^2}\,R}
\end{equation*}
for some absolute constant $C_{1}>0.$

\item If the sectional curvature $K$ of $M$ satisfies $K\geq -1,$ then

\begin{equation*}
A\left( R\right) \leq C_{1}\,e^{\frac{4\sqrt{1-H^2}}{\sqrt{7}}\,R}
\end{equation*}
for some absolute constant $C_{1}>0.$
\end{enumerate}
\end{theorem} 

In the case where $M = \mathbb{H}^2 \times \mathbb{R}$, we can obtain an improved constant in Theorem \ref{int1}. More precisely,

\begin{theorem} \label{thmH2R}
Let $\Sigma$ be a stable $H$-surface in $\mathbb{H}^2 \times \mathbb{R}$ with $H^2 < \dfrac{1}{2}$. Assuming that the geodesic ball $B_p(R)$ does not intersect the boundary of $\Sigma$ or the cut locus of $p$ in $\Sigma$, then
$$A(R)\leq C e^{aR},$$
where $a=\dfrac{4\sqrt{\dfrac{1}{2}-H ^2}}{\sqrt{7}}$ and $C$ is a constant absolute.

\end{theorem}

Do Carmo and Peng \cite{MP} proved that a complete, orientable, stable minimal surface in $\mathbb R^3$ is flat. This result was also independently proved by D. Fischer-Colbrie and Schoen \cite{FCS} and  Pogolerov \cite{russo}. The proof relies on the parabolicity of stable minimal surfaces and area estimates of geodesic balls, as shown by Pogorelov \cite{russo} and Colding-Minicozzi \cite{CM2}. An improved area estimate, provided by Munteanu, Sung, and Wang, allowed for a proof of this result without invoking the parabolicity of stable minimal surfaces. It is worth mentioning that Chodosh and Li \cite{CL} recently extended do Carmo and Peng's result to the case of dimension four.

On the other hand, according to the seminal work of R. Bryant \cite{Bryant} (also see \cite{CHR, KKMS, UY}), the geometry of minimal surfaces in Euclidean space $\mathbb{R}^3$ is very similar to the geometry of $1$-surfaces in hyperbolic $3$-space $\mathbb{H}^3$. This similarity allows us to consider the following area estimate for stable $1$-surfaces in the three-dimensional hyperbolic space $\mathbb{H}^3$.

\begin{theorem}
\label{int2} Let $\Sigma $ be a stable  $1$-surface in $\mathbb{H}^{3}$. 
Then there exists a universal constant $R_{0}$ such that for any geodesic
ball $B_{p}\left( R\right)$ with no intersection with the boundary of $\Sigma,$ it holds
\begin{equation}
L\left( r\right) \leq 2\pi r\left( 1+\frac{10}{\ln R}\right)
\end{equation}%
and 
\begin{equation}\label{eq29}
A\left( r\right) \leq \pi\, r^2\left( 1+\frac{10}{\ln R}\right)
\end{equation}%
for all $r\leq \sqrt{R}$ and $R\geq R_{0}$. Here, $L(r)$ denotes the length of geodesic circle $\partial B_p(r)$ and $A(r)$ denotes the area of $B_p(r)$.
\end{theorem}

As a direct application of the above result, we present a new proof of the following well-known result established by da Silveira \cite{Silveira}. 

\begin{corollary}\label{corint2}
Let $\Sigma $ be a complete noncompact stable $1$-surface in $\mathbb{H}^{3}$, then $$A(r)\leq \pi r^2$$ for all $r > 0$. In particular, this implies that $\Sigma$ is a horosphere.
\end{corollary}

Another relevant topic closely related to the area estimates of geodesic balls is the study of the spectrum of a complete manifold. We recall that the bottom spectrum of a complete manifold $\Sigma$ can be defined as the optimal constant in the Poincaré inequality

\begin{equation}
\lambda_0(\Sigma)\int_{\Sigma} \phi^2 \leq \int_{\Sigma} |\nabla \phi|^2,  
\end{equation}
where $\phi \in C^{\infty}_0(\Sigma)$.
According to Li and Wang \cite{Li}, we have
\begin{equation}\label{lnstimatelambda0}
\lambda _{0}(\Sigma) \leq \dfrac{1}{4}\left( \liminf _{R\rightarrow \infty} \dfrac{\ln A_{p}(R)}{R}  \right) ^2,
\end{equation}
where $A_{p}(R)$ denotes the area of a geodesic ball $B_{p}(R)$ centered at the point $p$ with radius $R$.

By applying the inequality \eqref{lnstimatelambda0} to a stable $H$-surface and using Theorem \ref{int1}, we can provide an upper estimate of the bottom spectrum. Specifically, we obtain the following:

\begin{corollary}
Let $\Sigma$ be a stable complete $H$-surface, with $|H| < 1$, in a three-dimensional Riemannian manifold $M$:
\begin{enumerate}
\item[\textbf{a)}] If the scalar curvature $S$ of $M$ satisfies $S\geq -6$, then
$$\lambda _{0}(\Sigma)\leq 1-H^2.$$
\item[\textbf{b)}] If the sectional curvature $K$ of $M$ satisfies $K\geq -1$, then
$$\lambda _{0}(\Sigma)\leq \dfrac{4-4H^2}{7}.$$
\end{enumerate}
\end{corollary}

To obtain an estimate for the bottom spectrum of $H$-surfaces in $\mathbb{H}^2 \times \mathbb{R}$, we can once again apply Theorem \ref{thmH2R} together with the inequality \eqref{lnstimatelambda0}. This leads to the following estimate

\begin{corollary}\label{corolario}
Let $\Sigma$ be a stable $H$-surface in $\mathbb{H}^2 \times \mathbb{R}$ with $H^2 < \dfrac{1}{2}$. Assuming that the geodesic ball $B_p(R)$ does not intersect the boundary of $\Sigma$ or the cut locus of $p$ in $\Sigma$, then
$$\lambda _0(\Sigma) \leq \dfrac{2-4H ^2}{7}.$$
\end{corollary}

We note that Corollary \ref{corolario} has already been obtained by Bérard, Castillon, and Cavalcante in \cite{Petrucio} using different techniques.

In higher dimensions, our approach does not work to obtain volume estimates for geodesic balls. However, by following the technique employed in Theorem 14 of \cite{Mun}, we have successfully estimated the bottom spectrum of a complete stable $H$-hypersurface immersed in a manifold with sectional curvature bounded from below. More precisely,

\begin{theorem}\label{comparasionthmupfour}
Let $\Sigma$ be a complete stable $H$-hypersurface in a $(n+1)$-manifold $M$, where $n=3$ or $n=4$. If the sectional curvature $K$ of $M$ satisfies $K \geq -1$, the following estimates hold:
\begin{enumerate}
    \item [\textbf{a)}] For $n=3$, if $|H|< \dfrac{\sqrt{10}}{3}$, then
    $$\lambda _0(\Sigma) \leq \frac{10-9H^2}{4};$$
    \item[\textbf{b)}] For $n=4$, if $|H|< \dfrac{\sqrt{7}}{2}$, then
    $$\lambda _0(\Sigma) \leq 24-12H^2.$$
\end{enumerate}
\end{theorem}

We should remark that the bounds on $|H|$ in Theorem \ref{comparasionthmupfour}  also play a significant role in \cite{zhou}. In particular, they appear in determining the number of ends of a weakly stable $H$-hypersurface in $\mathbb{H}^{n+1}$, with $n=3$ or $n=4$.

The organization of the paper is as follows. In the upcoming section, we will establish the notation and show some preliminary results. The last section will be dedicated to providing the proofs of the main results outlined in this introduction.\\

\noindent
{\bf Acknowledgments.} We would like to thank Professor Luciano Mari for his interest and related discussions.

\section{Preliminaries}

Let $x: \Sigma \rightarrow M$ be an isometric immersion of a surface $\Sigma$ in a Riemannian three-dimensional manifold $M$. Let $X: (-\varepsilon, \varepsilon) \times \Sigma \rightarrow M$ be a differentiable variation of $M$ where $X(0,p) = x(p)$ and $X(t,p) = X_t(p)$ is an immersion of $\Sigma$ in $M$ for all $|t|<\varepsilon$. For each $t \in (-\varepsilon,\varepsilon)$, we define the \textit{area function} $\mathcal{A}(t) = \mathrm{Area}(X_t)$ and the \textit{volume function} $\mathcal{V}(t)$ as

$$\mathcal{V}(t) = \int_{[0,t]\times \Sigma} X^*dV,$$
which measures the signed volume enclosed between $X_0 = x$ and $X_t$. Furthermore, we consider the functional
$$\mathcal{J}(t) = \mathcal{A}(t) - 2c\mathcal{V}(t), \ \ \ \mathrm{with} \ c \in \mathbb{R}.$$ By applying the first variation formulas for area and volume, we can deduce that $\Sigma$ is a critical point of $\mathcal{J}(t)$ if and only if $\Sigma$ has constant mean curvature $H = c$. Such a surface is commonly referred to as an $H$-surface. In this case, we introduce the \textit{Jacobi operator} $L$ on $\Sigma$, which is given by
$$L=\Delta +|h|^2+\mathrm{Ric(\nu , \nu )},$$
where $h$ is the second fundamental form of $\Sigma$, $\mathrm{Ric}$ is the Ricci curvature of $M$ and $\nu $ is the outward unit normal vector field. If $\Sigma$ is an $H$-surface, the second variation formula of the functional $\mathcal{J}(t)$ can be expressed as
\begin{eqnarray*}
   Q(f,f) & = & \dfrac{d^{2}}{dt^2}\Big|_{t=0}\mathcal{J}(t) = -\int _{\Sigma} fLfdA \\
    &=& \int _{\Sigma} \left[|\nabla f|^2-(|h|^2+\mathrm{Ric(\nu , \nu )})f^2\right]dA, \ \ \ \forall f \in C_{0}^{\infty}(\Sigma).
\end{eqnarray*}

An $H$-surface $\Sigma $ is said to be \textit{strongly stable} if $Q(f,f) \geq 0$. This concept is equivalent to the positiveness of the first eigenvalue of the Jacobi operator $L$ and the existence of a positive smooth solution $u$ for the equation $Lu = 0$ (see \cite{FC}). Note that if $\Sigma$ is a strongly stable $H$-surface in $M$, then
\begin{equation}
    \label{estabilidade}
    \int _{\Sigma} (|h|^2+\mathrm{Ric(\nu ,\nu )})f^2 \leq \int _{\Sigma} |\nabla f|^2, \ \ \ \forall f \in C_{0}^{\infty}(\Sigma).
\end{equation}
 We will refer to \eqref{estabilidade} as \textit{the stability inequality}. We say that a $CMC$ surface is \textit{weakly stable} if the stability inequality \eqref{estabilidade} holds for every test function $f \in C_{0}^{\infty}(\Sigma)$ such that $\int_{\Sigma} f = 0$. A strongly stable $H$-surface is also weakly stable. However, the converse may not be true in general. For instance, the standard embedding of $\mathbb{S}^2$ into $\mathbb{S}^3$ as a totally geodesic surface is weakly stable but not strongly stable.

In this work, we focus exclusively on the study of strongly stable $H$-surfaces, and for brevity, we refer to them simply as \textit{stable} $H$-surfaces.

 The stability of an $H$-surface $\Sigma$ is closely related to an upper bound on the bottom spectrum of the Laplacian on $\Sigma$, denoted by $\lambda_0(\Sigma)$. The bottom spectrum $\lambda_0(N)$ of a complete manifold $N$ can be characterized as the optimal constant in the Poincaré inequality
$$\lambda_0(N) \int_N f^2 \leq \int_N |\nabla f|^2, \ \ \ \forall f \in C^{\infty}_0(N). $$
Furthermore, according to \cite{Li}, for any $p \in N$,
\begin{equation}\label{eq14}
\lambda _{0}(N) \leq \dfrac{1}{4}\left( \liminf _{R\rightarrow \infty} \dfrac{\ln A_{p}(R)}{R}  \right) ^2,
\end{equation}
where $A_{p}(R)$ denotes the area of a geodesic ball $B_{p}(R)$ centered at $p$ with radius $R$.

 An important concept for complete Riemannian manifolds is the \textit{parabolicity}. We recall that a complete manifold $(N,g)$ is \textit{nonparabolic} if it admits a positive symmetric Green's function. Otherwise, we say that it is \textit{parabolic} (see \cite{LiTam}). It is well-known that $M$ is nonparabolic if and only if it admits a nonconstant bounded superharmonic function. Besides, another relevant result, obtained by Colding and Minicozzi \cite{CM}, is the following

\begin{theorem*}
\label{prop}
If $\Sigma$ is a complete surface so that for all $r > 0$ we have that a geodesic ball $B_p(r)$ in $\Sigma$ satisfies
$$A(r) = Area (B_{p}(r)) \leq C r^2,$$
then $\Sigma$ is parabolic.
\end{theorem*}

Now we fix $p\in \Sigma $ and let $B_{p}\left( R\right) =\left\{ x\in \Sigma :r\left( x\right) <R\right\}$ be the  intrinsic geodesic ball of radius $R$ in $\Sigma,$ where $
r\left( x\right) =d_{\Sigma }\left( p,x\right)$
 is the intrinsic distance on $\Sigma $. We denote by
\begin{eqnarray*}
L\left( r\right) &=&\int_{\partial B_{p}\left( r\right) }ds 
\end{eqnarray*}
and
\begin{eqnarray*}
A\left( r\right) &=&\int_{B_{p}\left( r\right) }dA,
\end{eqnarray*} 
the length of the geodesic circle $\partial B_{p}\left( r\right)$ and the
    area of $B_{p}\left( r\right)$, respectively.
By the first variation of the length of a curve, we get
\begin{eqnarray}\label{substituir10}\nonumber
\dfrac{d}{dr}L(r)&=&\dfrac{d}{dr}\left( \int_{\partial B_{p}\left( r\right) }ds\right)\\
&=&\int_{\partial B_{p}(r)} \langle \kappa _{g}\nu,\nu \rangle ds  = \int _{\partial B_{p}(r)} \kappa _{g}ds,
\end{eqnarray}
where $\kappa _{g}$ is the geodesic curvature of $\partial B_{p}(r).$

Also, we let $S$, $\mathrm{Ric}$, and $K$ denote the scalar, Ricci, and sectional curvatures of $M$, respectively, while $S_{\Sigma}$ and $K_{\Sigma}$ will be the scalar and Gauss curvatures of $\Sigma$, respectively. Tracing the Gauss equation, we have
\begin{equation}
S_{\Sigma}=S-2\mathrm{Ric}(\nu ,\nu)+4 H^2-|h|^2.
\end{equation} 
Since on a surface $S_{\Sigma}=2K_{\Sigma}$, the equation above becomes
\begin{equation}\label{eq00}
K_{\Sigma}=\frac{S}{2}-\mathrm{Ric}(\nu, \nu)+ 2H^2-\frac{1}{2}|h|^2.
\end{equation} 

Let $\{ e_1,e_2\}$ be a local orthonormal frame in $\Sigma$ and $R_{1212}$ be the sectional curvature of $M$ for the $2$-plane spanned by $\{e_1,e_2\}$, we have
$$R_{1212}+\mathrm{Ric}(\nu ,\nu)=\frac{1}{2}S.$$
Then, using \eqref{eq00}, we obtain
\begin{equation}
\label{eqseclem}
-2K_{\Sigma}=-2R_{1212}+|h|^2-4H^2.
\end{equation}

The following lemma will play a crucial role in deriving the area estimates for stable $H$-surfaces in three-dimensional spaces.

\begin{lemma}
\label{Main} Let $\Sigma $ be a stable $H$-surface in a three dimensional
manifold $M.$ Let $B_{p}\left( R\right) $ be a geodesic ball in $\Sigma $
that does not intersect the cut locus of $p$ in $\Sigma $ or the boundary of 
$\Sigma.$ Assume that $\phi=\phi \left( r\right) $ is a Lipschitz
continuous, nonincreasing function on $\left[ 0,R\right] $ with $\phi \left(
R\right) =0.$

\begin{itemize}
\item[\textbf{a)}] If the scalar curvature of $M$ satisfies $S\geq -6\alpha$ for some $%
\alpha\geq 0,$ then

\begin{equation}
-2\int_{0}^{R}\phi \left( r\right) \phi ^{\prime }\left( r\right) L^{\prime
}\left( r\right) dr\leq 2\pi \phi ^{2}\left( 0\right) +\int_{B_{p}\left(
R\right) }\left( \phi ^{\prime }\right) ^{2}+3(\alpha-H^2)\,\int_{B_{p}\left(
R\right) }\phi ^{2};  \label{Main1}
\end{equation}

\item[\textbf{b)}] If the sectional curvature of $M$ satisfies $K\geq -\alpha$ for some $%
\alpha\geq 0,$ then

\begin{equation}
-4\int_{0}^{R}\phi \left( r\right) \phi ^{\prime }\left( r\right) L^{\prime
}\left( r\right) dr\leq 4\pi \phi ^{2}\left( 0\right) +\int_{B_{p}\left(
R\right) }\left( \phi ^{\prime }\right) ^{2}+4(\alpha-H^2)\,\int_{B_{p}\left(
R\right) }\phi ^{2}.  \label{Main2}
\end{equation}
\end{itemize}
\end{lemma}

\begin{proof}

For any $0<r<R$, by applying the Gauss-Bonnet theorem in \eqref{substituir10} and using that the Euler characteristic $\chi (B_{p}(r)) \leq 1$, we have

\begin{eqnarray}\label{eq02}
\dfrac{d}{dr}L(r)&=&\int _{\partial B_{p}(r)} \kappa _{g}ds = 2\pi \chi (B_{p}(r)) - \int _{B_{p}(r)} K_{\Sigma} dA \nonumber \\
&\leq & 2\pi - \int _{B_{p}(r)} K_{\Sigma} dA. \label{eqp01}
\end{eqnarray}
For item \textbf{a)}, assuming $S \geq -6\alpha$, we can deduce from equation \eqref{eq00} that
\begin{eqnarray*}
K_{\Sigma}&\geq & -3\alpha -\mathrm{Ric}(\nu , \nu)-\dfrac{1}{2}|h|^{2}+2H^2 \\
&=&-3\alpha -(\mathrm{Ric}(\nu , \nu)+|h|^2)+\dfrac{1}{2}|h|^2+2H^2\\
&=&3(\alpha -H^2)+(\mathrm{Ric}(\nu , \nu) + |h|^2),
\end{eqnarray*}
where we have used in the last line that $|h|^2 \geq 2H^2$. Integrating the above inequality over $B_{p}(r)$, we obtain
$$-\int_{B_{p}(r)} K_{\Sigma} \leq 3(\alpha -H^2 )A(r)+\int_{B_{p}(r)}(\mathrm{Ric}(\nu , \nu)ds+|h|^2) .$$
Then \eqref{eqp01} becomes
\begin{equation}
L'(r) \leq 2\pi +\int_{B_{p}(r)}(\mathrm{Ric}(\nu , \nu) + |h|^2) +3(\alpha -H^2)A(r). \label{eqp02}
\end{equation}
Multiplying \eqref{eqp02} by $-2\phi(r)\phi'(r)$ (which is nonnegative) and integrating from $r=0$ to $r=R$, we have
\begin{eqnarray}\label{peq03}
-2\int_{0}^{R} \phi(r)\phi'(r)L'(r)dr &\leq & 2\pi \phi ^2(0) -2\int_{0}^{R} \phi (r)\phi '(r)\left( \int_{B_{p}(r)}(\mathrm{Ric}(\nu , \nu)+|h|^2)\right)dr \nonumber \\
& &-6(\alpha -H^2)\int_{0}^{R}\phi (r)\phi '(r)A(r)dr.
\end{eqnarray}
Note that for any real function $f(r)$ with $f(0)=0$, we have

\begin{equation}\label{eqp04}
-2\int_{0}^{R}\phi (r) \phi '(r)f(r)dr=\int_{0}^{R} f'(r)\phi ^2(r)dr.
\end{equation}
If we take $f(r)=\int_{B_{p}(r)}dA=A(r)$, we can apply \eqref{eqp04} to obtain

\begin{eqnarray} \label{eqphida}
-2\int_{0}^{R}\phi '(r)\phi (r)A(r)dr&=&\int_{0}^{R} \phi ^2(r)A'(r)dr \nonumber \\
&=&\int_{0}^{R} \phi ^2(r)L(r)dr=\int_{B_{p}(R)} \phi^2(r),
\end{eqnarray}
where we have used the coarea formula in the last equation. Therefore, we obtain

\begin{equation}
-6(\alpha -H^2) \int_{0}^{R}\phi (r)\phi '(r)A(r)=
3(\alpha-H^2) \int_{B_{p}(R)}\phi ^2.
\end{equation}

Now, substituting $f(r)=\int_{B_{p}(r)}(\mathrm{Ric}(\nu , \nu)+|h|^2)$ into \eqref{eqp04}, we have

\begin{eqnarray}\label{eqppta}\nonumber
-2\int_{0}^{R}\phi (r)\phi '(r)\left( \int_{B_{p}(r)}(\mathrm{Ric}(\nu , \nu)+|h|^2)\right)&=&\int_{0}^{R}\phi ^2(r)\left(   \int_{\partial B_{p}(r)}(\mathrm{Ric}(\nu , \nu)+|h|^2)\right)\\ \nonumber
&=&\int_{B_{p}(R)}(\mathrm{Ric}(\nu ,\nu)+|h|^2)\phi ^2\\
&\leq & \int_{B_p(R)}(\phi ')^2, 
\end{eqnarray}
where we have used the stability inequality \eqref{estabilidade}. Then, \eqref{peq03} becomes

$$-2\int_{0}^{R}\phi (r)\phi '(r)L'(r)dr \leq 2\pi \phi ^2(0)+ \int_{B_{p}(R)}(\phi ')^2+3(\alpha-H^2)\int_{B_{p}(R)}\phi ^2.$$ 
This concludes the proof of item a).\\

For item \textbf{b)}, we observe that $K\geq -\alpha$ implies $-2R_{1212}\leq 2\alpha$. Thus, \eqref{eqseclem} becomes
$$-2K_{\Sigma} \leq 2(\alpha -2H^2) + |h|^2.$$
Using the last expression in \eqref{eq02}, we have
\begin{eqnarray*}
2L'(r) &\leq & 4\pi +2(\alpha-2H^2) A(r)+\int_{B_{p}(r)}|h|^2.
\end{eqnarray*}
Multiplying the above expression  by $-2\phi (r)\phi '(r)$ (which is nonnegative), and integrating from $r=0$ to $r=R$, we get

\begin{eqnarray}\label{eq06}
-4\int_{0}^{R}\phi (r)\phi '(r)L'(r)dr&\leq &4\pi \phi ^2(0) -4(\alpha-2H^2)\int_{0}^{R}\phi (r)\phi '(r)A(r)dr \nonumber\\
& &-2\int_{0}^{R}\phi (r)\phi '(r)\left( \int_{B_{p}(r)}|h|^2 \right)dr.
\end{eqnarray}
Note that $\mathrm{Ric}(\nu , \nu)=K(\nu , e_1)+K(\nu ,e_2) \geq -2\alpha $, since  $K\geq -\alpha$.
Then, integrating over $B_{p}(r)$:
$$\int _{B_{p}(r)} \mathrm{Ric}(\nu ,\nu)\geq -2\alpha A(r).$$
Since  $-2\phi(r) \phi '(r)\geq 0$, we have
\begin{equation}\label{eq07}
-2\int _{0}^{R} \phi (r)\phi '(r)\int _{B_{p}(r)}\mathrm{Ric}(\nu , \nu) \geq 4\alpha \int _{0}^{R}\phi (r)\phi '(r)A(r).
\end{equation}
Furthermore, we obtain

$$-2\int _{0}^{R}\phi (r)\phi '(r)\left( \int _{B_{p}(r)}(\mathrm{Ric}(\nu , \nu)+|h|^2) \right)dr \leq \int _{B_{p}(r)}(\phi ')^2.$$
Applying \eqref{eq07} and \eqref{eqphida}, we get

\begin{eqnarray} \label{sub1}
-2\int _{0}^{R}\phi (r)\phi '(r)\left(\int _{B_{p}(r)}|h|^2\right)dr&\leq & \int _{B_{p}(R)}(\phi ')^2+2\alpha \int _{B_{p}(R)}\phi ^2.
\end{eqnarray}
 Moreover, using \eqref{eqphida} again, we have

\begin{eqnarray} \label{sub2} 
    -4(\alpha -2H^2)\int _{B_{p}(r)}\phi (r)\phi '(r)A(r)&=&2(\alpha -2H^2)\int _{B_{p}(r)}\phi ^2.
\end{eqnarray}
Finally, replacing \eqref{sub1} and \eqref{sub2} into \eqref{eq06}, we obtain

$$-4\int_{0}^{R}\phi(r)\phi^{'}(r)L^{'}(r)dr\leq 4\pi \phi^{2}(0)+\int_{B_(p)(R)} (\phi^{'})^2 + (4\alpha -4H^2)
\int_{B_{p}(R)} \phi^2.$$
This is the stated expression \eqref{Main1}.

\end{proof}

Next, assuming that $\phi = \phi(r)$ is of class $C^2$, we get 

\begin{corollary}
\label{Corlemma} Let $\Sigma $ be a stable $H$-surface in a three-dimensional
manifold $M.$ Let $B_{p}\left( R\right) $ be a geodesic ball in $\Sigma $
that does not intersect the cut locus of $p$ in $\Sigma $ or the boundary of 
$\Sigma.$ Assume that $\phi=\phi \left( r\right) $ is a $C^{2}$
nonincreasing function on $\left[ 0,R\right] $ with $\phi \left( R\right)
=0. $

\begin{itemize}
\item[\textbf{a)}] If the scalar curvature of $M$ satisfies $S\geq -6\alpha$ for some $%
\alpha\geq 0,$ then 
\begin{equation}\label{substituir11}
\int_{B_{p}\left( R\right) }\left( \phi ^{\prime }\right)
^{2}+2\int_{B_{p}\left( R\right) }\phi \phi ^{\prime \prime }\leq 2\pi \phi
^{2}\left( 0\right) +3(\alpha-H^2)\,\int_{B_{p}\left( R\right) }\phi ^{2};
\end{equation}

\item[\textbf{b)}] If the sectional curvature of $M$ satisfies $K\geq -\alpha$ for some $%
\alpha\geq 0,$ then 
\begin{equation}\label{substituir12}
3\int_{B_{p}\left( R\right) }\left( \phi ^{\prime }\right)
^{2}+4\int_{B_{p}\left( R\right) }\phi \phi ^{\prime \prime }\leq 4\pi \phi
^{2}\left( 0\right) +4(\alpha-H^2)\,\int_{B_{p}\left( R\right) }\phi ^{2}.
\end{equation}
\end{itemize}
\end{corollary}

\begin{proof}
Using integration by parts, we have

\begin{eqnarray*}
-2\int _{0}^{R}\phi (r)\phi '(r)L'(r)dr&=&2\int _{0}^{R}[(\phi ' (r))^2+\phi (r)\phi ''(r)]L(r)dr \\
&=&2\int _{B_{p}(R)}(\phi ' )^2+2\int _{B_{p}(R)}\phi \phi '',
\end{eqnarray*}
where we have used the coarea formula. To conclude, we can substitute this expression in \eqref{Main1} and \eqref{Main2}, respectively.
    
\end{proof}

\section{Proofs of main results}

\subsection{Proof of Theorem \ref{int1}}
We begin by proving item b) of the theorem. By item b) of
Lemma \ref{Main}, taking $\alpha = 1$, for any Lipschitz continuous, nonincreasing function $\phi = \phi(r)$ on $\left[ 0,t\right] $ with $\phi \left(
t\right) =0$, we have

\begin{equation}\label{eq08}
-4\int _{0}^{t}\phi (r)\phi '(r)L'(r)dr\leq 4\pi \phi ^2(0)+\int _{B_{p}(t)}(\phi ')^2  + (4-4H^2)\int _{B_{p}(t)}\phi ^2,
\end{equation}
for each $0 < t < R$. 
Consider $\displaystyle \phi (r)=e^{-\frac{b}{2}r}\psi (r)$, where $\psi$ is a Lipschitz continuous, nonincreasing function such that $\psi (t)=0$ and $b=\frac{4\sqrt{1-H^2}}{\sqrt{7}}$.
By direct computations, the left side of inequality \eqref{eq08} becomes

\begin{eqnarray*}
-4\int _{0}^{t} \phi (r) \phi '(r)L'(r)dr&=&\int _{0}^{t}(2b\psi ^2(r)-4\psi (r) \psi '(r))e^{-br}L'(r)dr  \\
&=&2b\int _{0}^{t}e^{-br}\psi ^2(r)L'(r)dr-4\int _{0}^{t}e^{-br}\psi (r)\psi '(r)L'(r)dr.
\end{eqnarray*}
After integration by parts and applying the coarea formula, we obtain that the first term on the right side is given by

$$2b\int _{0}^{t}e^{-br}\psi ^2(r)L'(r)dr=\int _{B_{p}(t)}(2b^2\psi ^2-4b\psi \psi ')e^{-br}.$$
Substituting these last expressions into \eqref{eq08}, we get

\begin{eqnarray*}
-4\int _{0}^{t}e^{-br}\psi (r)\psi '(r)L'(r)dr&\leq &4\pi \psi ^2(0)+\left( 4-4H^2-\frac{7b^2}{4}\right) \int _{B_{p}(t)}\psi ^2e^{-br} \\
& & + \, 3b \int _{B_{p}(t)}\psi \psi 'e^{-br}+\int _{B_{p}(t)}(\psi ')^2e^{-br}. 
\end{eqnarray*}
Recall that $b=\frac{4\sqrt{1-H^2}}{\sqrt{7}}$, then the above expression becomes

\begin{equation}\label{sub01}
-4\int _{0}^{t}e^{-br}\psi (r)\psi '(r)L'(r)dr\leq 4\pi \psi ^2(0)+\int _{B_{p}(t)}(\psi ')^2e^{-br}+3b\int _{B_{p}(t)}\psi \psi 'e^{-br}.
\end{equation}
From now on, the proof follows the same lines as in Theorem 10 of \cite{Mun}. For the sake of completeness, we provide the details below. For any $\eta$ with $0<\eta<R$ and $\eta\leq t$, consider

\begin{equation} \label{substituir04}
\psi(r)=
\left\{
  \begin{array}{rcl}
1,& \mathrm{if} & r\leq t-\eta\\
\frac{t-r}{\eta},& \mathrm{if} & t-\eta <r<t\\
0,&\mathrm{if} & r\geq t.\\
\end{array}
  \right. 
 \end{equation} 
 Then, by \eqref{sub01}, we obtain
\begin{eqnarray} \label{substituir01} \nonumber
 \frac{4}{\eta}\int_{t-\eta}^te^{-br}\psi(r)L'(r)dr &\leq & 4\pi+\frac{1}{\eta^2}\int_{B_p(t)\setminus B_p(t-\eta)}e^{-br}\\
 & & -\frac{3b}{\eta}\int_{B_p(t)\setminus B_p(t-\eta)}\psi(r)e^{-br}.
\end{eqnarray}
The integral on the left side is equivalent to

\begin{eqnarray} \label{substituir02} \nonumber
  \frac{4}{\eta}\int_{t-\eta}^te^{-br}\psi(r)L'(r)dr&=& -\frac{4}{\eta}L(t-\eta)e^{-b(t-\eta)}  
  +\frac{4b}{\eta}\int_{B_p(t)\setminus B_p(t-\eta)}\psi e^{-br}\\
  &&+\frac{4}{\eta^2}\int_{B_p(t)\setminus B_p(t-\eta)}e^{-br}.
\end{eqnarray}
Combining \eqref{substituir01} and \eqref{substituir02}, we have

\begin{equation}\label{0002}
\frac{3}{\eta^2}\int_{B_p(t)\setminus B_p(t-\eta)}e^{-br}+\frac{7b}{\eta}\int_{B_p(t)\setminus B_p(t-\eta)}\psi e^{-br}\leq 4\pi+\frac{4}{\eta}L(t-\eta)e^{-b(t-\eta)},
\end{equation}
for all $0<\eta\leq t<R.$ In particular, letting $\eta=t$, we get

$$\dfrac{3}{t^2}\int _{B_p(t)}e^{-br}\leq 4\pi -\dfrac{7b}{t}\int _{B_p(t)}e^{-br} \leq 4\pi.$$
Therefore, we have

\begin{equation} \label{substituir03}
 \int_{B_p(t)}e^{-br}\leq \frac{4\pi}{3}t^2,
\end{equation}
where the last expression holds for $0\leq r\leq t$. This implies that 

$$
\int_{B_p(t)}e^{-bt}\leq  \int_{B_p(t)}e^{-br}\leq \frac{4\pi}{3}t^2.
$$
Thus,
 \begin{equation} \label{substituir05}
     A(t)\leq \frac{4\pi}{3}t^{2}e^{bt}, \quad \forall t\leq R. 
 \end{equation}
Next, we claim that there exists an absolute constant $\Lambda>0$ such that for any $\tau$ and $s$ satisfying $0<2\tau\leq s<R-3\tau,$ then

\begin{equation}\label{0003}
\int_{B_p(s)\setminus B_p(s-\tau)} e^{-br}\leq \Lambda\tau +\frac{\Lambda}{\tau}\int_{B_p(s-\tau)\setminus B_p(s-2\tau)} e^{-br}.
\end{equation}
Indeed, letting $\eta=4\tau$ and $T=s-\frac{3\tau}{2}$, we have
$$\frac{\eta}{8}\leq T<R-\frac{9\eta}{8}.$$
Applying the mean value theorem to the function $f(t)=\int _{B_p(t)}e^{-br}$, there exists $\xi\in \left(T-\frac{\eta}{8},T+\frac{\eta}{8}\right)$ such that

\begin{equation}
\label{meanvaluethm2}
\int_{B_p(T+\frac{\eta}{8})}e^{-br}-\int_{B_p(T-\frac{\eta}{8})}e^{-br}=f'(\xi)\left(T+\frac{\eta}{8}-T+\frac{\eta}{8}\right),
\end{equation}
in other words, 

\begin{equation}\label{0004}
\int_{B_p(T+\frac{\eta}{8})\setminus B_p(T-\frac{\eta}{8})} e^{-br}=\frac{\eta}{4}e^{-b\xi}L(\xi).
\end{equation}
Letting  $t=\xi+\eta\in\left(T-\frac{\eta}{8}+\eta,T+\frac{\eta}{8}+\eta\right)=\left(T+\frac{7\eta}{8},T+\frac{9\eta}{8}\right)$, and using \eqref{substituir04}, we have

$$
\int_{B_p(t)\setminus B_p(t-\eta)} \psi e^{-br}\geq \frac{1}{2} \int_{B_p(t-\frac{\eta}{2})\setminus B_p(t-\eta)} e^{-br}\geq \frac{1}{2}\int_{B_p(T+\frac{3\eta}{8})\setminus B_p(T+\frac{\eta}{8})} e^{-br}.
$$
Multiplying the above inequality by $\frac{7b}{2\eta}$ and using \eqref{0002} and \eqref{0004}, we get

\begin{eqnarray}\label{eq0005}
\frac{7b}{2\eta}\int_{B_p(T+\frac{3\eta}{8})\setminus B_p(T+\frac{\eta}{8})} e^{-br}&\leq & \frac{7b}{\eta}\int_{B_p(t)\setminus B_p(t-\eta)} \psi e^{-br}\nonumber\\
&\leq &4\pi+\frac{4}{\eta}L(t-\eta)e^{-b(t-\eta)}-\frac{3}{\eta^2}\int_{B_p(t)\setminus B_p(t-\eta)} e^{-br}\nonumber\\
&\leq &4\pi+\frac{4}{\eta}L(t-\eta)e^{-b(t-\eta)}\nonumber\\
&=& 4\pi +\frac{16}{\eta^2}\int_{B_p(T+\frac{\eta}{8})\setminus B_p(T-\frac{\eta}{8})} e^{-br}.
\end{eqnarray}
Therefore, the above expression becomes

$$
\int_{B_p(T+\frac{3\eta}{8})\setminus B_p(T+\frac{\eta}{8})} e^{-br}\leq \frac{8\pi\eta}{7b}+\frac{32}{7b\eta}\int_{B_p(T+\frac{\eta}{8})\setminus B_p(T-\frac{\eta}{8})} e^{-br}.
$$
From this, we can infer that exists $\Lambda>0$ such that 
$$
\int_{B_p(T+\frac{3\eta}{8})\setminus B_p(T+\frac{\eta}{8})} e^{-br}\leq \Lambda\eta+\frac{\Lambda}{\eta}\int_{B_p(T+\frac{\eta}{8})\setminus B_p(T-\frac{\eta}{8})} e^{-br}.
$$
Finally, substituting $\eta = 4\tau$ and $s = T + \frac{3\tau}{2}$ above, we can conclude that \eqref{0003} holds, as we claimed.\\
Now, let $s \geq 6\Lambda$, taking $\tau = 2\Lambda$ in \eqref{0003} and iterating \eqref{0003} $m$ times with $m = \left[\frac{s}{2\Lambda}\right] - 2$, we obtain

$$
\int_{B_p(s)\setminus B_p(s-2\Lambda)} e^{-br}\leq 2\Lambda^2\sum_{k=0}^{m-1}\frac{1}{2^k}+\frac{1}{2^m}\int_{B_p(6\Lambda)} e^{-br}.
$$
Since the above series is convergent, we can deduce from \eqref{substituir05} that for $6\Lambda \leq s \leq R - 6\Lambda$, there exists a constant $C_2$ such that

\begin{equation}\label{eq0006}
\int_{B_p(s)\setminus B_p(s-2\Lambda)} e^{-br}\leq C_2.
\end{equation}
Using the mean value theorem as in \eqref{meanvaluethm2}, we can find $\xi \in (R - 16\Lambda, R - 14\Lambda)$ such that

\begin{equation}\label{eq0007}
\int_{B_p(R-14\Lambda)\setminus B_p(R-16\Lambda)} e^{-br}=2\Lambda L(\xi)e^{-b\xi}.
\end{equation}
Applying \eqref{eq0006} with $s=R-14\Lambda$, we can infer from \eqref{eq0007} that  
$$
L(R-\eta)e^{-b(R-\eta)}\leq C_3,
$$ for some positive constant $C_3$, and $\eta = R- \xi \in (14\Lambda, 16\Lambda)$. Plugging this into \eqref{0002}, we get
$$
\frac{3}{\eta^2}\int_{B_p(R)\setminus B_p(R-\eta)}e^{-br}\leq 4\pi +\dfrac{4}{\eta} L(R-\eta) e^{-b(R-\eta)}\leq C_4.
$$
Then, since  $r\leq R$, we have
$$
\frac{3}{\eta^2}\int_{B_p(R)\setminus B_p(R-\eta)}e^{-bR}\leq\frac{3}{\eta^2}\int_{B_p(R)\setminus B_p(R-\eta)}e^{-br}\leq C_4,
$$
which implies that 
$$
\int_{B_p(R)\setminus B_p(R-\eta)}dA\leq C_5e^{bR},
$$
for some $\eta\in(14\Lambda,16\Lambda)$. In particular,

\begin{equation}\label{0008}
\int_{B_p(R)\setminus B_p(R-14\Lambda)}dA\leq C_5e^{bR}.
\end{equation}
Replacing $R$ by $R-14k\Lambda$ in \eqref{0008}, where $k=1,2,\cdots, n$ and $n=\left[\frac{R}{14\Lambda}\right]-1,$ we obtain 

\begin{equation}
\label{eqquadraticestimate}
A(R)\leq \sum_{k=0}^{n}\int_{B_p(R-14k\Lambda)\setminus B_p(R-14(k+1)\Lambda)}dA+A(14\Lambda).
\end{equation}
By \eqref{0008}, we have

\begin{eqnarray*}
A(R)&\leq & C_6\sum_{k=0}^{n}e^{b(R-14k\Lambda)}+C_6\\
&\leq & Ce^{bR},
\end{eqnarray*}
where $C$ is an absolute constant that bounds the sum. This finishes the proof of item b).\newline

We will now proceed with the proof of item a). Recall that we are assuming $S \geq -6$. Then, by \eqref{Main1} of Lemma \ref{Main}, we get

\begin{equation}\label{eq10}
-2\int_{0}^{t}\phi(r)\phi^{'}(r)L^{'}(r)dr\leq 2\pi \phi^{2}(0)+\int_{B_{p}(t)} (\phi^{'})^2 + (3 -3H^2)\int_{B_{p}(t)} \phi^2.
\end{equation}
Setting $\phi (r) =e^{-ar}\psi (r)$, where $\psi$ is a nonincreasing Lipschitz function on $\left[ 0,t \right ]$ such that $\psi
(t)=0$, we have

\begin{eqnarray} \label{substitui06} 
\nonumber
2\int _{0}^{t}(a\psi^2(r) - \psi(r)\psi'(r))e^{-2ar}L'(r)dr& \leq & 2\pi \psi ^2(0) + \int _{B_{p}(t)}e^{-2ar}((\psi ')^2 - 2a\psi\psi') \\ 
& &+\left(a^2 + 3-3H^2 \right) \int _{B_{p}(t)} e^{-2ar}\psi ^2.
\end{eqnarray}
Integrating by parts, we obtain

\begin{eqnarray*}
2a\int _{0}^{t}e^{-2ar} \psi ^2(r)L'(r)dr
&=&-4a\int _{0}^{t}e^{-2ar}\psi (r)\psi '(r)L(r)dr+4a^2\int _{0}^{t}e^{-2ar}\psi ^2(r)L(r) dr.
\end{eqnarray*}
By coarea formula this is equivalent to

$$2a\int _{0}^{t} e^{-2ar}\psi ^2(r)L'(r)dr =-4a\int _{B_p(t)}\psi \psi 'e^{-2ar}+4a^2\int _{B_p(t)}\psi ^2e^{-2ar}.$$
Replacing this back into \eqref{substitui06} and setting $a=\sqrt{1-H^2}$, we obtain that

$$-2\int _{0}^{t} e^{-2ar}\psi (r)\psi '(r)L'(r)dr\leq  2\pi \psi ^2(0) +2a\int _{B_{p}(t)}\psi \psi 'e^{2ar} + \int _{B_{p}(t)}e^{-2ar}(\psi ')^2.$$
Using the function $\psi$ defined by \eqref{substituir04}, we have

\begin{eqnarray}\label{substituir07}
    \dfrac{2}{\eta }\int _{t-\eta}^{t}e^{-2ar}\psi (r)L'(r)dr&\leq &2\pi -\dfrac{2a}{\eta }\int _{B_p(t)\setminus B_p(t-\eta )}e^{-2ar}\psi \\ \nonumber
    &+&\dfrac{1}{\eta ^2}\int _{B_p(t)\setminus B_p(t-\eta )}e^{-2ar}.
\end{eqnarray}
Integrating by parts and using the coarea formula, we arrive at

\begin{equation*}
    \int _{t-\eta}^{\eta }\psi e^{-2ar}L'=-e^{-2a(t-\eta )}L(t-\eta )+\dfrac{1}{\eta }\int _{B_p(t)\setminus B_p(t-\eta )}e^{-2ar}+2a\int _{B_p(t)\setminus B_p(t-\eta )}\psi e^{-2ar}
\end{equation*}
Plugging these on \eqref{substituir07}, we get

\begin{equation*}
    \dfrac{1}{\eta ^2}\int _{B_p(t)\setminus B_p(t-\eta )}e^{-2ar}+\dfrac{6a}{\eta }\int _{B_p(t)\setminus B_p(t-\eta )}\psi e^{-2ar}\leq 2\pi +\dfrac{2}{\eta }e^{-2a(t-\eta )}L(t-\eta )
\end{equation*}
From this point on, we can continue with the same calculations as in item b) to derive that

$$A(R) \leq Ce^{2aR}.$$
\qed

Now, we will use similar arguments to the Lemma \ref{Main} and to the Theorem \ref{int1} to estimate the area of a geodesic ball in $\mathbb{H}^2 \times \mathbb{R}$.

\subsection{Proof of Theorem \ref{thmH2R}}

Let $\nu$ be the unit normal vector field along $\Sigma$, and $\partial _t$ the unit tangent vector field to $\Sigma$ at $\mathbb{R}$ direction in $\mathbb{H}^2\times \mathbb{R}$. By Gauss equation, we have
\begin{equation}\label{KH2R}
K_{\Sigma}=-v^2-\dfrac{1}{2}|h|^2+2H^2,
\end{equation}
where $v=\langle \nu , \partial _t  \rangle $ is the component of $\nu$ with respect to $\partial _t$. 
Applying \eqref{KH2R} into \eqref{eq02}, we get

\begin{eqnarray} \label{eq22}
L'(r)&\leq & 2\pi+\int _{B_p(r)} \left( v^2 -2H^2+\dfrac{1}{2} |h|^2 \right).
\end{eqnarray}
Consider a Lipschitz continuous, nonincreasing function $\phi = \phi(r)$ on $\left[ 0,t\right] $ with $\phi \left(
t\right) =0$. Multiplying \eqref{eq22} by $-4\phi (r)\phi '(r)$ (which is nonnegative) and integrating from $r=0$ to $r=R$, we obtain

\begin{eqnarray}\label{corrigir}\nonumber
    -4\int _{0}^{R} \phi (r)\phi '(r)L'(r)dr&\leq &4\pi (\phi (0))^2-4\int _{0}^{R}\phi (r)\phi '(r)\left( \int _{B_p(r)}v^2\right)dr \\
    &&+8H^2\int _{0}^{R}\phi (r)\phi '(r)A(r)dr
    -2\int _{0}^{R}\phi (r)\phi '(r)\left( \int _{B_p(r)}|h|^2  \right)dr.
\end{eqnarray}
Now we will analyze each term of the above expression. Applying \eqref{eqp04} to $f(r)=\int _{B_p(r)}v^2$ and using the coarea formula, we have

\begin{equation} \label{nova}
    -4\int _{0}^{R}\phi (r)\phi '(r)\int _{B_p(r)}v^2=2\int _{B_p(R)}v^2\phi ^2(r)
\end{equation}
Using \eqref{eqphida}, we get

\begin{eqnarray} \label{eq23}
\;\;\;\;\;\;8H^2\int _{0}^{R} \phi (r)\phi '(r)A(r)&=&-4H^2\left( -2\int _{0}^{R}\phi (r)\phi '(r)A(r)  \right) \\ \nonumber
&=&-4H^2\int _{B_p(R)}(\phi (r))^2.
\end{eqnarray}
Moreover, by \eqref{eqppta}, we have
\begin{eqnarray} \label{eq24}
    -2\int _{0}^{R}\phi (r)\phi '(r)\int _{B_p(r)}(\mathrm{Ric}(\nu , \nu)+|h|^2)&\leq & \int _{B_p(R)}(\phi '(r))^2.
\end{eqnarray}
Using \eqref{nova}, \eqref{eq23}, and \eqref{eq24} in \eqref{corrigir}, we obtain

\begin{eqnarray}\label{substituir09} \nonumber 
-4\int_{0}^{R}\phi (r)\phi '(r)L'(r)dr &\leq & 4\pi \phi^2(0)+\int _{B_p(R)} (v^2+1) \phi ^2\\
& & -4H^2\int _{B_p(R)} \phi ^2+\int_{B_p(R)}(\phi')^2,
\end{eqnarray}
where we have used that $\mathrm{Ric}(\nu ,\nu )=v^2+1$ in $\mathbb{H}^2 \times \mathbb{R}$. Since  $|v|^2 \leq 1$, we obtain
\begin{eqnarray} \label{nova2}
-4\int _{0}^{R}\phi (r)\phi '(r)L'(r) dr&\leq  &4\pi \phi ^2(0)+(2-4H^2)\int _{B_p(R)} \phi ^2\\ \nonumber
& &+\int _{B_p(R)}(\phi ')^2.
\end{eqnarray}
Defining $\phi(r) =e^{-\frac{a}{2}r}\psi (r)$, where $\psi $ is a Lipschitz  continuous nonincreasing function with $\psi (R)=0$, then \eqref{nova2} becomes

\begin{eqnarray} \label{adicionada}\nonumber 
    2\int _{0}^{R}(a\psi ^2(r)-2\psi(r)\psi'(r))e^{-ar}L'(r)dr &\leq &4\pi \psi ^2(0) + \left( 2-4H^2-\frac{7a^2}{4} \right) \int _{B_p(R)}e^{-ar}\psi ^2\\
    & &+\int _{B_p(R)}e^{-ar}(\psi ')^2-a\int _{B_p(R)}e^{-ar}\psi \psi '.
\end{eqnarray}
By integration by parts and applying the coarea formula, we obtain

$$2a\int _{0}^{R}e^{-ar}\psi ^2(r)L'(r)dr=2a^2\int _{B_p(R)}e^{-ar}\psi ^2-4a\int _{B_p(R)}e^{-ar}\psi \psi '.$$

$$-4\int _{0}^{R} e^{-ar} \psi (r)\psi '(r)L'(r)dr \leq  4\pi \psi ^2(0)+\int _{B_p(R)} (\psi ')^2e^{-ar} +3a\int _{B_p(R)} \psi \psi 'e^{-ar}.$$
From this point on, arguing as in the proof of Theorem \ref{int1}, we obtain the desired inequality.
\qed

\subsection{Proof of Theorem \ref{int2}}

Letting $\alpha = 1$ and $H=1$ in \eqref{substituir12}, we obtain 

\begin{equation} 
3\int _{B_{p}(R)}(\phi ')^2 +4\int _{B_{p}(R)}\phi \phi ''\leq 4\pi \phi (0)^2,
\end{equation}
where $\phi = \phi (r)$ is a $C^2$ nonincreasing function in $[0,R]$ with $\phi (R)=0$. In particular, we have

\begin{equation} \label{eq11}
\int _{B_{p}(R)}\phi \phi ''\leq \pi \phi^2 (0).
\end{equation}
Taking $\phi =\ln \; (R+1)-\ln \; (r+1)$, this expression becomes

\begin{equation*}
   \pi \; \ln^2(R+1) \geq \int _{B_{p}(R)}\dfrac{\ln\; (R+1)-\ln\;(r+1)}{(r+1)^2} \geq \int _{0}^{R}\dfrac{\mathrm{ln}\; (R+1)-\mathrm{ln}\;(r+1)}{(r+1)^2}L(r)dr,
\end{equation*}
where we used the coarea formula.
This is equivalent to

\begin{eqnarray}\label{eq12} \nonumber
    \pi \; \mathrm{ln}^2(R+1) &\geq & \int _{0}^{R}\dfrac{\mathrm{ln}\; (R+1)-\mathrm{ln}\;(r+1)}{(r+1)^2}(L(r)-2\pi r)dr \\
    & &+2\pi \int _{0}^{R}\dfrac{\mathrm{ln}\; (R+1)-\mathrm{ln}\;(r+1)}{(r+1)^2}rdr.
\end{eqnarray}
By the Gauss equation, we have

$$K_{\Sigma}=-1-\dfrac{1}{2}|h|^2+2H^2.$$
Since $H= 1$ and $|h|^2\geq 2H^2$, we obtain $K_{\Sigma}\leq 0$ in $B_p(R)$. By the Hessian comparison theorem, we have

 \begin{equation}\label{eq13}
     2\pi \leq \dfrac{L(r)}{r}\leq \dfrac{L(R)}{R}, \ \ \ \ \mathrm{for \ all} \  0 < r < R.
 \end{equation}
Now, the proof follows the same lines as Theorem 9 of \cite{Mun}. However, we will continue the proof for the sake of completeness.

By contradiction, suppose that

 \begin{equation}\label{eq25}
\dfrac{L(r)}{r}\geq 2\pi \left( 1+\dfrac{10}{\mathrm{ln}(R+1)}\right),     
 \end{equation}
for all $r \in [\sqrt{R}, R]$. Substituting this into \eqref{eq12}, we get 

\begin{eqnarray} \label{desigualdade} \nonumber
    \pi \; \mathrm{ln}^2(R+1) &\geq &\dfrac{20\pi}{\mathrm{ln} (R+1)}\int _{\sqrt{R}}^{R}\dfrac{\mathrm{ln}(R+1)-\mathrm{ln}(r+1)}{(r+1)^2}rdr\\
    &+&2\pi \int _{0}^{R}\dfrac{\mathrm{ln}(R+1)-\mathrm{ln}(r+1)}{(r+1)^2}rdr
\end{eqnarray}
Note that the second term on the right-hand side becomes

\begin{equation*}
    2\pi \int _{0}^{R}\dfrac{\mathrm{ln}(R+1)-\mathrm{ln}(r+1)}{(r+1)^2}rdr = -2\pi \ln (R+1)+\pi \ln ^2(R+1)+2\pi -\dfrac{2\pi} {R+1}.
\end{equation*}
Replacing this in \eqref{desigualdade}, we find

\begin{eqnarray}\label{eq26} \nonumber 
    2\pi \; \mathrm{ln}(R+1) &\geq &
    \dfrac{20\pi }{\mathrm{ln}(R+1)}\int _{\sqrt{R}}^{R}\dfrac{\mathrm{ln}(R+1)-\mathrm{ln}(r+1)}{(r+1)^2}rdr +2\pi -\dfrac{2\pi} {R+1} \\
    &\geq &
    \dfrac{20\pi }{\mathrm{ln}(R+1)}\int _{\sqrt{R}}^{R}\dfrac{\mathrm{ln}(R+1)-\mathrm{ln}(r+1)}{(r+1)^2}rdr,
\end{eqnarray}
On the other hand, the first term of \eqref{desigualdade} can be bounded as

\begin{eqnarray*}
  \int _{\sqrt{R}}^{R}\dfrac{\mathrm{ln}(R+1)-\mathrm{ln}(r+1)}{(r+1)^2}rdr &= &-\mathrm{ln}\left( \dfrac{R+1}{\sqrt{R}+1} \right)\left( \mathrm{ln}(\sqrt{R}+1)+\dfrac{1}{\sqrt{R}+1} \right)+\dfrac{1}{2}\mathrm{ln}^2(R+1)\\
  &-&\dfrac{1}{2}\mathrm{ln}^2(\sqrt{R}+1)+\dfrac{1}{\sqrt{R}+1}-\dfrac{1}{R+1}\\
  &\geq &\dfrac{1}{9}\mathrm{ln}^2(R+1),
\end{eqnarray*}
for all $R>R_0$ sufficiently large. Substituting this into \eqref{eq26}, we obtain

$$\mathrm{ln}(R+1) \geq \dfrac{10}{9}\mathrm{ln}(R+1).$$
Therefore, we obtain a contradiction, showing that there exists $r_0 \in [\sqrt{R}, R]$ such that

$$\dfrac{L(r_0)}{r_0} \leq 2\pi \left( 1+\dfrac{10}{\mathrm{ln}(R+1)} \right).$$
Thus, for $r < \sqrt{R}$, using \eqref{eq13}, we obtain

\begin{equation*}
L (r) \leq 2\pi r\left( 1+\frac{10}{\mathrm{ln} \;R}\right).
\end{equation*}
Hence, after integrating from $0$ to $r$, we conclude that
\begin{equation*}
A(r) \leq \pi\, r^2\left( 1+\frac{10}{\mathrm{ln} \; R}\right).
\end{equation*}
\qed

\subsection{Proof  of  Corollary \ref{corint2}.}

Letting $R\rightarrow \infty$ in \eqref{eq29}, we obtain 
 $$A(r)\leq \pi r^2.$$
 Thus, according to Proposition $1.37$ of \cite{CM}, we conclude that $\Sigma$ is parabolic.
Furthermore, it is well known (see Theorem 1 of \cite{FCS}) that there exists a positive function $g$ such that
 
 \begin{equation}\label{corolarioeq}
 \Delta g+\mathrm{Ric}(\nu , \nu )g+|h|^2g=0,
 \end{equation}
then
\begin{eqnarray*}
\Delta g&=&2g-|h|^2g\\
&\leq & 2(1-H^2)g,
\end{eqnarray*}
where we have used that $Ric(\nu,\nu) = -2$ and $|h|^2 \geq 2H^2$. Since we are assuming that $H=1$, we obtain that $\Delta g \leq 0.$  As  $\Sigma$ is parabolic, we conclude that $g$ is constant. From this and the equation \eqref{corolarioeq} we can deduce that $\Sigma$ is umbilic. Since it is noncompact, we conclude that $\Sigma$ is a horosphere.\\
\qed

\subsection{Proof of Theorem \ref{comparasionthmupfour}.}

We will need the following result

\begin{lemma}\label{teoremadegreen}	

Let $M$ be a complete $(n+1)$-dimensional manifold with sectional curvature bounded from below and $n\leq 4$. For a complete stable $H$-hypersurface $\Sigma$ in $M$ with $\lambda _0(\Sigma) > 0$, its minimal positive Green's function $G(x)= G(p,x)$ satisfies 

$$\int _{\Sigma \setminus B_p(1)} \dfrac{|\nabla G|^4}{G^3\ln ^{2q}(1+G^{-1})} < \infty ,$$
for any $q > \frac{1}{2}$.

\end{lemma}

\begin{proof}

Set $v= \mathrm{ln}\;G$. Using 
$$\Delta v=-|\nabla v|^2$$
and
$$|\nabla v|_{11}=\langle \nabla |\nabla v|,\nabla v \rangle$$
where $\{ e_1,e_2,...,e_n \}$ is a local orthonormal frame on $\Sigma $ with $e_1=\dfrac{\nabla v}{|\nabla v|}$.
Applying Bochner's formula, we have 

\begin{equation}\label{eq20}
\frac{1}{2}\Delta |\nabla v|^2 \geq \frac{1}{n-1}|\nabla v|^4+\frac{n}{n-1}|\nabla |\nabla v||^2-\frac{n-2}{n-1}\langle \nabla |\nabla v|^2, \nabla v\rangle + \mathrm{Ric}_{\Sigma} (\nabla v, \nabla v)
\end{equation}
on $\Sigma \setminus \{ p\}$. 

Let $(h_{ij})$ denote the second fundamental form of $\Sigma$. By the Gauss equation,

\begin{eqnarray}\label{eqric} \nonumber
    Ric_{\Sigma}(e_1,e_1)&=&\sum _{\alpha =2}^nK(e_1,e_1)+\sum _{\alpha =2}^n h_{11}h_{\alpha \alpha}-\sum _{\alpha =2}^n h_{1\alpha}^2 \\
    &=&\sum _{\alpha =2}^nK(e_1,e_1) +h_{11}nH-h_{11}^2-\sum _{\alpha =2}^n h_{1\alpha}^2\\ \nonumber
    &=&\sum _{\alpha =2}^nK(e_1,e_1)+h_{11}nH-\sum _{\alpha =1}^n h_{1\alpha}^2.
\end{eqnarray}
By applying Proposition 3.2 of \cite{zhou}, we get

\begin{equation}\label{desiric}
   h_{11}nH-\sum _{\alpha =1}^n h_{1\alpha}^2\geq -|h|^2+\dfrac{n^2(5-n)H^2}{4}. 
\end{equation}
Since  the sectional curvature is bounded from below, there exists a constant $C_0>0$ such that
\begin{equation}\label{desicurvatura}
\sum _{\alpha =2}^{n}K(e_1,e_1)\geq -C_0(n-1). 
\end{equation}
Applying \eqref{desiric} and \eqref{desicurvatura} to \eqref{eqric}, we obtain
\begin{eqnarray*}
\mathrm{Ric}_{\Sigma}(\nabla v, \nabla v)&\geq & C_0(n-1)|\nabla v|^2-|h|^2|\nabla v|^2+\dfrac{n^2(5-n)H^2|\nabla v|^2}{4}\\
& = & C_0[(n-1)+\mathrm{Ric}(\nu , \nu)]|\nabla v|^2-(\mathrm{Ric}(\nu , \nu)+|h|^2)|\nabla v|^2\\
& & +\dfrac{n^2(5-n)H^2|\nabla v|^2}{4}.
\end{eqnarray*}
Note that there exists a constant $C$ such that
$-C_0[(n-1)+\mathrm{Ric}(\nu , \nu)] \geq -C$ , because the sectional curvature is bounded from below. So, the above inequality  becomes 
\begin{eqnarray}
\label{eqlemmaric}
\mathrm{Ric}_{\Sigma}(\nabla v, \nabla v) &\geq & -C|\nabla v|^2-(\mathrm{Ric} (\nu , \nu)+|h|^2)|\nabla v|^2 \\ \nonumber
&+&\dfrac{n^2(5-n)H^2|\nabla v|^2}{4}.
\end{eqnarray}
By the stability inequality, we have
$$\int _{\Sigma}(\mathrm{Ric}(\nu , \nu)+|h|^2)|\nabla v|^2\phi^2 \leq \int _{\Sigma}|\nabla (|\nabla v|\phi )|^2,$$
where $\phi$ is a cut-off function. 
Thus, \eqref{eqlemmaric} implies that
{\begin{equation}\label{eq30}
\int _{\Sigma} \mathrm{Ric} _{\Sigma}(\nabla v,\nabla v)\phi ^2 \geq -C\int _{\Sigma}|\nabla v|^2\phi ^2-\int _{\Sigma} |\nabla (|\nabla v|\phi)|^2 +\dfrac{n^2(5-n)H^2}{4}\int _{\Sigma}|\nabla v|^2\phi ^2  
\end{equation}
By the above inequality and \eqref{eq20}, we have
\begin{eqnarray} \label{eq31} \nonumber
   \dfrac{1}{n-1}\int _{\Sigma}|\nabla v|^4\phi ^2 &\leq & \dfrac{1}{2}
   \int _{\Sigma} \phi ^2 \Delta |\nabla v|^2+C\int _{\Sigma} 
   \phi ^2|\nabla v|^2 \\ 
   & &+\int _{\Sigma} |\nabla (|\nabla v|\phi)|^2 -\dfrac{n^2(5-n)H^2}{4}
   \int _{\Sigma}|\nabla v|^2\phi ^2 \\ \nonumber
   & &-\dfrac{n}{n-1}\int _{\Sigma}|\nabla |\nabla v||^2\phi^2+\dfrac{n-2}{n-1}\int _{\Sigma}\phi ^2\langle \nabla |\nabla v|^2,\nabla v \rangle.
\end{eqnarray}
Using the identity  
\begin{eqnarray} \label{eq33}
    |\nabla (|\nabla v|\phi)|^2&=&\phi ^2|\nabla |\nabla v||^2 +2\phi \langle \nabla|\nabla v|^2,|\nabla v|\nabla \phi \rangle +|\nabla v|^2|\nabla \phi |^2,
\end{eqnarray}
}
we see that  \eqref{eq31} is equivalent to
\begin{eqnarray*} 
    \dfrac{1}{n-1}\int _{\Sigma} |\nabla v|^4\phi^2&\leq &-\int _{\Sigma }\phi \langle \nabla \phi ,\nabla |\nabla v|^2 \rangle -\dfrac{n}{n-1} \int _{\Sigma} \phi ^2|\nabla |\nabla v||^2 \\ 
    & &+\dfrac{n-2}{n-1} \int _{\Sigma}\phi ^2\langle \nabla |\nabla v|^2,\nabla v \rangle + 
    C\int _{\Sigma} \phi ^2|\nabla v|^2+\int _{\Sigma} \phi ^2|\nabla |\nabla v||^2 \\
    & &+ \int _{\Sigma}2\phi \langle \nabla |\nabla v|,|\nabla v|\nabla \phi \rangle + \int _{\Sigma} |\nabla v|^2|\nabla \phi |^2\\
    & &-\dfrac{n^2(5-n)H^2}{4}\int _{\Sigma} \phi ^2|\nabla v|^2.
\end{eqnarray*}
Since $\langle \nabla \phi ,\nabla |\nabla v|^2 \rangle =2|\nabla v|\langle \nabla \phi ,\nabla |\nabla v| \rangle,$
we obtain
\begin{eqnarray} \label{eq34} \nonumber
    \dfrac{1}{n-1}\int _{\Sigma} |\nabla v|^4\phi^2&\leq &-\dfrac{1}{n-1} \int _{\Sigma}\phi ^2|\nabla |\nabla v||^2 + \dfrac{n-2}{n-1} \int _{\Sigma} \phi ^2 \langle  \nabla |\nabla v|^2, \nabla v \rangle \\
    &&+C\int _{\Sigma}\phi ^2|\nabla v|^2+\int _{\Sigma}|\nabla v|^2|\nabla \phi |^2 -\dfrac{n^2(5-n)H^2}{4}\int _{\Sigma} \phi ^2|\nabla v|^2.
\end{eqnarray}

Now, let $\phi =G^{\frac{1}{2}}\eta $, where $\eta$ is a cut-off function on  $\Sigma$ with $\eta = 0$ on $B_p(1)$. For any $\delta > 0$, we have

\begin{equation}\label{eq35}
    \int _{\Sigma} |\nabla v|^2 |\nabla \phi |^2 \leq C(\delta)\int _{\Sigma}|\nabla v|^2|\nabla \eta |^2G+\left(  \dfrac{1+\delta}{4} \right)\int _{\Sigma}\phi ^2|\nabla v|^4.
\end{equation}
Putting \eqref{eq35} into \eqref{eq34}, we arrive at
\begin{eqnarray*}
    \left( \dfrac{1}{n-1} -\dfrac{1+\delta}{4} \right)\int _{\Sigma} \phi ^2|\nabla v|^4 &\leq &-\dfrac{1}{n-1}\int _{\Sigma}\phi ^2|\nabla |\nabla v||^2+\dfrac{n-2}{n-1}\int _{\Sigma}\phi ^2\langle \nabla |\nabla v|^2,\nabla v \rangle \\
    &&+C\int _{\Sigma}\phi ^2|\nabla v| ^2+C(\delta)\int _{\Sigma}|\nabla v|^2|\nabla \eta |^2G\\
    &&-\dfrac{n^2(5-n)H^2}{4}\int _{\Sigma}\phi ^2|\nabla v|^2.
\end{eqnarray*}
Recalling that $v=\ln G$ and $G >0$, we have  $\phi ^2|\nabla v|^2=\dfrac{\eta ^2|\nabla G|^2}{G}$. Thus, the above expression  becomes
\begin{eqnarray*}
  \left(  \dfrac{1}{n+1}-\dfrac{1+\delta}{4} \right)\int _{\Sigma}\phi ^2|\nabla v|^4&\leq &C(\delta)\int _{\Sigma} (\eta ^2+|\nabla \eta |^2)G|\nabla v|^2 -\dfrac{1}{n-1}\int _{\Sigma}G\eta ^2|\nabla |\nabla v||^2 \\
  &+&\dfrac{n-2}{n-1}\int _{\Sigma}\eta ^2\langle \nabla |\nabla v|^2,\nabla G \rangle -\dfrac{n^2(5-n)H^2}{4}\int _{\Sigma}\dfrac{\eta ^2|\nabla G|^2}{G}. 
\end{eqnarray*}
Since  $G$ is harmonic, we have $\mathrm{div}(\eta ^2 \nabla G)=\langle  \nabla G,\nabla \eta ^2 \rangle$.  Thus,
$$-\int _{\Sigma}|\nabla v|^2\langle  \nabla G,\nabla \eta ^2 \rangle =\int _{\Sigma}\langle  \eta ^2 \nabla G,\nabla |\nabla v|^2 \rangle.$$
Therefore,
\begin{eqnarray} \label{eq36} \nonumber
    \left(  \dfrac{1}{n-1}-\dfrac{1+\delta}{4} \right)\int _{\Sigma}\phi ^2|\nabla v|^4 &\leq &C(\delta)\int _{\Sigma}(\eta ^2+|\nabla \eta |^2)G|\nabla v|^2-\dfrac{1}{n-1}\int _{\Sigma}G\eta ^2|\nabla |\nabla v||^2\\
    & &-\dfrac{n-2}{n-1}\int _{\Sigma}|\nabla v|^2\langle \nabla G,\nabla \eta ^2 \rangle -\dfrac{n^2(5-n)H^2}{4}\int _{\Sigma}\dfrac{\eta ^2|\nabla G|^2}{G}.
\end{eqnarray}
Using Young's  inequality, we have

\begin{equation*}
    -\dfrac{n-2}{n-1}\int _{\Sigma}|\nabla v|^2\langle  \nabla G,\nabla \eta ^2 \rangle \leq \delta \int _{\Sigma} |\nabla v|^4\phi ^2+C(\delta)\int _{\Sigma}|\nabla \eta |^2G|\nabla G|^2.
\end{equation*}
Substituting this into \eqref{eq36}, we obtain
\begin{eqnarray*}
    \left( \dfrac{1}{n-1}-\dfrac{1+\delta}{4} \right)\int _{\Sigma}\phi ^2|\nabla v|^4 &\leq &C(\delta)\int _{\Sigma}(\eta ^2+|\nabla \eta |^2)G|\nabla v|^2-\dfrac{1}{n+1}\int _{\Sigma}G\eta ^2|\nabla |\nabla v||^2\\
    & &+\delta \int _{\Sigma}|\nabla v|^4\phi ^2+C(\delta)\int _{\Sigma}|\nabla \eta |^2G|\nabla v|^2\\
    & &-\dfrac{n^2(5-n)H^2}{4}\int_{\Sigma}\dfrac{\eta ^2|\nabla G|^2}{G}.
\end{eqnarray*}
Consequently,
$$\left( \dfrac{1}{n-1}-\dfrac{1+\delta}{4}-\delta \right)\int _{\Sigma}|\nabla v|^4\phi ^2 \leq C(\delta)\int _{\Sigma}(\eta^2+|\nabla \eta|^2)G|\nabla v|^2.$$
Since  $n\leq 4$, we can choose $\delta =\delta (n)>0$ such that $\dfrac{1}{n-1}-\dfrac{1+\delta}{4}-\delta >0$. Then, there exists an absolute constant $\Gamma >0$ such that

\begin{eqnarray} \label{eq37} \nonumber
    \int _{\Sigma}|\nabla v|^4\phi ^2&\leq &\Gamma \int _{\Sigma}(\eta ^2+|\nabla \eta |^2)G|\nabla v|^2\\
    &=&\Gamma \left(  \int _{\Sigma}\eta ^2G|\nabla v|^2+\int _{\Sigma}|\nabla \eta |^2G|\nabla v|^2\right) .
\end{eqnarray}
For $\frac{1}{2}<q<1,$ let $\eta=\psi\,w(G),$ where $\psi$ is a cut-off function
such that $\psi=0$ on $B_p(1)\cup \left(M\setminus B_p(2R)\right),$ $\psi=1$ on $B_p(R)\setminus B_p(2),$
and
 
\begin{equation*}
w\left( G\right) =\frac{1}{\ln ^{q}\left( A\,G^{-1}\right) }
\end{equation*}
with $A=e^{2\sqrt{\Gamma}}\alpha,$ $\alpha=\max_{\partial B_p(1)} G.$ Since  $|\nabla v|^2=\frac{|\nabla G|^2}{G^2}$, after some computations, we obtain 
\begin{eqnarray}\label{parte1} \nonumber
    \int _{\Sigma}\eta ^2G|\nabla v|^2&=&\int _{\Sigma}\psi ^2\omega ^2G\dfrac{|\nabla G|^2}{G^2}\\
    &\leq &\int _{L(0,\alpha)}\dfrac{|\nabla G|^2}{G\ln ^{2q}(AG^{-1})}\leq C.
\end{eqnarray}
Besides, we have 

\begin{eqnarray*}
    \int _{\Sigma}|\nabla \eta |^2G|\nabla v|^2&\leq &2\int _{\Sigma}\dfrac{|\nabla \psi |^2|\nabla G|^2}{G\ln^{2q}(AG^{-1})}+2\int _{\Sigma}\psi ^2|\nabla \omega |^2G|\nabla v|^2\\
    &\leq &C+2\int _{\Sigma}\psi ^2|\nabla \omega |^2G|\nabla v|^2.
\end{eqnarray*}
Furthermore, note that
\begin{equation*}
    2\int _{\Sigma}\psi ^2|\nabla \omega |^2G|\nabla v|^2 =2\int _{\Sigma}\dfrac{q^2}{[\ln (AG^{-1})]^2}\phi ^2|\nabla v|^4\leq \dfrac{1}{2\Gamma }\int _{\Sigma}\phi ^2|\nabla v|^4
\end{equation*}
So, we obtain
\begin{equation} \label{parte2}
   \int _{\Sigma}|\nabla \eta |^2G|\nabla v|^2 \leq C +\dfrac{1}{2\Gamma }\int _{\Sigma}\phi ^2|\nabla v|^4.
\end{equation}
Plugging \eqref{parte1} and \eqref{parte2} into \eqref{eq37}, we obtain
\begin{eqnarray*}
    \int _{\Sigma}|\nabla v |^4\phi^2 &\leq &C+\frac{1}{2}\int _{\Sigma}\psi ^2|\nabla v|^4\\
    &\leq & C.
\end{eqnarray*}
This is equivalent to
$$\int _{B_p(R) \setminus B_p(2)} \dfrac{|\nabla G|^4}{G^3 \ln ^{2q}(AG^{-1})} < C.$$
Letting $R \rightarrow \infty$,  we find

$$\int _{\Sigma \setminus B_p(2)} \dfrac{|\nabla G|^4}{G^3 \ln ^{2q}(AG^{-1})} < \infty.$$
Proving the desired result.
\end{proof}
\textbf{Proof of Theorem \ref{comparasionthmupfour}:}

Without loss of generality, let's assume $\lambda _0(\Sigma)>0$.
Let $\phi = |\nabla G|^{\frac{1}{2}}\varphi$.  By the Poincaré inequality, we have
$$\lambda _0(\Sigma)\int_{\Sigma}|\nabla G|\varphi ^2\leq \int _{\Sigma} |\nabla (|\nabla G|^{\frac{1}{2}}\varphi)|^2.$$
Using Cauchy's inequality on the right-hand side of the previous expression, we obtain
\begin{eqnarray}\label{eq16}
\lambda _0(\Sigma)\int _{\Sigma}|\nabla G|\varphi ^2&\leq & \left( \frac{1}{4} 
+\delta \right)\int _{\Sigma}\varphi ^2|\nabla |\nabla G||^2|\nabla G|^{-1} \\
&+&C(\delta)\int _{\Sigma}|\nabla G||\nabla \varphi |^2, \nonumber
\end{eqnarray}
 for any $\delta >0$. Since $G$ is a harmonic  away from $p$, by Kato's inequality (see \cite{Yau}), we have that 
 $$|\nabla ^2G|^2 \geq \dfrac{n}{n-1}|\nabla |\nabla G||^2.$$
Applying this inequality  into the Bochner's formula, we obtain
 $$\Delta |\nabla G|\geq \dfrac{1}{n-1}|\nabla |\nabla G||^2 |\nabla G|^{-1} + \mathrm{Ric}_{\Sigma}(\nabla G,\nabla G)|\nabla G|^{-1}.$$
 Arguing similarly as in \eqref{eqlemmaric}, we have
 $$\mathrm{Ric}_{\Sigma} (\nabla G, \nabla G) \geq -(n-1)|\nabla G|^2-|h|^2|\nabla G|^2+\dfrac{n^2(5-n)H^2|\nabla G|^2}{4}.$$ Thus, we get
 \begin{eqnarray} \label{eq17}
 \int _{\Sigma}\Delta |\nabla G| \varphi ^2 &\geq & \dfrac{1}{n-1}\int _{\Sigma} |\nabla |\nabla G||^2 |\nabla G|^{-1} \varphi ^2 - (n-1)\int _{\Sigma} |\nabla G| \varphi ^2 \\
 & &-\int _{\Sigma} |h|^2|\nabla G| \varphi ^2 +\dfrac{n^2(5-n)H^2}{4}\int _{\Sigma} |\nabla G| \varphi ^2. \nonumber
 \end{eqnarray}
 Since the sectional curvature is bounded, we have that $\mathrm{Ric} (\nu , \nu) \geq -n$. Then,  the stability inequality becomes
 $$\int _{\Sigma} |h|^2|\nabla G| \varphi ^2 \leq n\int _{\Sigma} |\nabla G| \varphi ^2 + \int _{\Sigma} |\nabla (|\nabla G|^{\frac{1}{2}} \varphi)|^2.$$ 
 Expanding the last term of the above expression, we have
 \begin{eqnarray*}
 \int _{\Sigma} |h|^2|\nabla G| \varphi ^2 &\leq &  \dfrac{1}{4}\int _{\Sigma} \varphi ^2|\nabla |\nabla G||^2 |\nabla G|^{-1} + \dfrac{1}{2}\int _{\Sigma}\langle \nabla |\nabla G|,\nabla \varphi ^2  \rangle\\
 & &+ \int _{\Sigma} |\nabla G| |\nabla \varphi |^2 + n\int _{\Sigma}|\nabla G|\varphi ^2. \end{eqnarray*}
  Plugging the above expression  into \eqref{eq17}, we get 
 
 \begin{eqnarray}  \label{eq18} 
  \dfrac{1}{n-1} \int _{\Sigma} |\nabla |\nabla G||^2 |\nabla G|^{-1} \varphi ^2 &\leq & \int _{\Sigma} \Delta |\nabla G| \varphi ^2 +\dfrac{1}{4}\int _{\Sigma} |\nabla |\nabla G||^2 |\nabla G|^{-1} \varphi ^2 \nonumber\\ 
 & &+ \dfrac{1}{2} \int _{\Sigma} \langle \nabla |\nabla G|, \nabla \varphi ^2  \rangle +\int_{\Sigma} |\nabla G| |\nabla \varphi |^2\\
 & &+\left(2n-1-\dfrac{n^2(5-n)H^2}{4} \right)\int _{\Sigma} |\nabla G|\varphi ^2. \nonumber
 \end{eqnarray}
 
 By the first Green's identity,
 \eqref{eq18} becomes
 
 \begin{eqnarray} \label{eq19}
 \left( \dfrac{1}{n-1} -\dfrac{1}{4}  \right) \int _{\Sigma} |\nabla |\nabla G||^2 |\nabla G|^{-1} \varphi ^2 &\leq & -\dfrac{1}{2}\int _{\Sigma} \langle \nabla \varphi ^2, \nabla |\nabla G|  \rangle+\int _{\Sigma} |\nabla G| |\nabla \varphi |^2 \nonumber \\ 
 &&  +\left( 2n-1-\dfrac{n^2(5-n)H^2}{4}  \right)\int _{\Sigma} |\nabla G| \varphi ^2. 
\end{eqnarray}
Furthermore, it is straightforward to check that
 $$|\langle \nabla \varphi ^2, \nabla |\nabla G|  \rangle|\leq 2\left(  \dfrac{\varphi |\nabla |\nabla G||}{|\nabla G|^{\frac{1}{2}}} . |\nabla \varphi | |\nabla G|^{\frac{1}{2}} \right)  \leq 2\left( \dfrac{2\epsilon \varphi ^2|\nabla |\nabla G||^2}{|\nabla G|} + \dfrac{|\nabla \varphi| ^2 |\nabla G|}{4\epsilon}  \right).$$
 Taking $2\epsilon = \delta$, we arrive at
 \begin{eqnarray*}
 \left( \dfrac{1}{n-1} -\dfrac{1}{4}+\delta \right)\int _{\Sigma} |\nabla |\nabla G||^2 |\nabla G|^{-1} \varphi ^2 &\leq& C(\delta) \int _{\Sigma} |\nabla \varphi |^2 |\nabla G|\\
 & &+ \left( 2n-1 -\dfrac{n^2 (5-n)H^2}{4}  \right)\int _{\Sigma} |\nabla G| \varphi ^2.
 \end{eqnarray*}
 Moreover, since  $ \dfrac{1}{n-1} -\dfrac{1}{4}+\delta \geq 0$, we have
\begin{eqnarray*}
\int _{\Sigma} |\nabla |\nabla G||^2 |\nabla G|^{-1} \varphi ^2 &\leq & C(\delta) \int _{\Sigma} |\nabla \varphi |^2 |\nabla G| +  2n-1 \\
& &-\dfrac{4(n-1)(8n-4-n^2 (5-n)H^2)}{4(5-n+4(n-1)\delta)}  \int _{\Sigma} |\nabla G| \varphi ^2.
\end{eqnarray*}
 Putting all of these facts together yields the following inequality
 
 $$\left[ \lambda _0(\Sigma)-\left( \dfrac{1}{4}+\delta  \right)\dfrac{4(n-1)(8n-4-n^2(5-n)H^2)}{4(5-n+4(n-1)\delta)}  \right] \int _{\Sigma} |\nabla G| \varphi ^2 \leq C(\delta) \int _{\Sigma} |\nabla \varphi|^2|\nabla G|.$$
 Finally, by  Lemma \ref{teoremadegreen}, we can conclude that
 \begin{equation}
 \label{eqeigenvalue}
 \lambda _0 (\Sigma) \leq \dfrac{(n-1)(8n-4-n^2(5-n)H^2)}{20-4n}.
\end{equation} 
In order to obtain the stated result, it is enough to substitute $n=3$ and $n=4$ in \eqref{eqeigenvalue}, respectively.
\\
\qed

\bibliographystyle{plain}
\bibliography{references.bib}

\begin{thebibliography}{10}

\bibitem{barbosa1988stability}
J~Lucas Barbosa, Manfredo do~Carmo, and Jost Eschenburg.
\newblock Stability of hypersurfaces of constant mean curvature in riemannian
  manifolds.
\newblock {\em Mathematische Zeitschrift}, 197:123--138, 1988.

\bibitem{LM}
Jo\~{a}o~Lucas Barbosa and Manfredo do~Carmo.
\newblock Stability of hypersurfaces with constant mean curvature.
\newblock {\em Math. Z.}, 185(3):339--353, 1984.

\bibitem{Petrucio}
Pierre B\'{e}rard, Philippe Castillon, and Marcos Cavalcante.
\newblock Eigenvalue estimates for hypersurfaces in {$\Bbb H^m\times \Bbb R$}
  and applications.
\newblock {\em Pacific J. Math.}, 253(1):19--35, 2011.

\bibitem{Bryant}
Robert Bryant.
\newblock Surfaces of mean curvature one in hyperbolic space.
\newblock {\em Ast{\'e}risque}, 154(155):321--347, 1987.

\bibitem{zhou}
Xu~Cheng, Leung-fu Cheung, and Detang Zhou.
\newblock The structure of weakly stable constant mean curvature hypersurfaces.
\newblock {\em Tohoku Math. J. (2)}, 60(1):101--121, 2008.

\bibitem{CL}
Otis Chodosh and Chao Li.
\newblock Stable minimal hypersurfaces in $\mathbb{R}^4$.
\newblock {\em To appear in Acta Math. arXiv preprint arXiv:2108.11462}, 2021.

\bibitem{CM2}
Tobias~H Colding and William~P Minicozzi.
\newblock Estimates for parametric elliptic integrands.
\newblock {\em International Mathematics Research Notices}, 2002(6):291--297,
  2002.

\bibitem{CM}
Tobias~Holck Colding and William~P. Minicozzi, II.
\newblock {\em A course in minimal surfaces}, volume 121 of {\em Graduate
  Studies in Mathematics}.
\newblock American Mathematical Society, Providence, RI, 2011.

\bibitem{CHR}
Pascal Collin, Laurent Hauswirth, and Harold Rosenberg.
\newblock The geometry of finite topology bryant surfaces.
\newblock {\em Annals of mathematics}, 153(3):623--659, 2001.

\bibitem{Silveira}
Alexandre~M. Da~Silveira.
\newblock Stability of complete noncompact surfaces with constant mean
  curvature.
\newblock {\em Math. Ann.}, 277(4):629--638, 1987.

\bibitem{MP}
M.~do~Carmo and C.~K. Peng.
\newblock Stable complete minimal surfaces in {${\bf R}^{3}$} are planes.
\newblock {\em Bull. Amer. Math. Soc. (N.S.)}, 1(6):903--906, 1979.

\bibitem{FC}
D.~Fischer-Colbrie.
\newblock On complete minimal surfaces with finite {M}orse index in
  three-manifolds.
\newblock {\em Invent. Math.}, 82(1):121--132, 1985.

\bibitem{FCS}
Doris Fischer-Colbrie and Richard Schoen.
\newblock The structure of complete stable minimal surfaces in {$3$}-manifolds
  of nonnegative scalar curvature.
\newblock {\em Comm. Pure Appl. Math.}, 33(2):199--211, 1980.

\bibitem{KKMS}
Nicholas~J Korevaar, Rob Kusner, William~H Meeks, and Bruce Solomon.
\newblock Constant mean curvature surfaces in hyperbolic space.
\newblock {\em American Journal of Mathematics}, 114(1):1--43, 1992.

\bibitem{LiTam}
Peter Li and Luen-Fai Tam.
\newblock Symmetric {G}reen's functions on complete manifolds.
\newblock {\em Amer. J. Math.}, 109(6):1129--1154, 1987.

\bibitem{Li}
Peter Li and Jiaping Wang.
\newblock Complete manifolds with positive spectrum.
\newblock {\em J. Differential Geom.}, 58(3):501--534, 2001.

\bibitem{MPR}
William~H. Meeks, III, Joaqu\'{\i}n P\'{e}rez, and Antonio Ros.
\newblock Stable constant mean curvature surfaces.
\newblock In {\em Handbook of geometric analysis. {N}o. 1}, volume~7 of {\em
  Adv. Lect. Math. (ALM)}, pages 301--380. Int. Press, Somerville, MA, 2008.

\bibitem{mori1983stable}
Hiroshi Mori.
\newblock Stable complete constant mean curvature surfaces in $\mathbb{R}^3$
  and $\mathbb{H}^3$.
\newblock {\em Transactions of the American Mathematical Society},
  278(2):671--687, 1983.

\bibitem{Mun}
Ovidiu Munteanu, Chiung-Jue~Anna Sung, and Jiaping Wang.
\newblock Area and spectrum estimates for stable minimal surfaces.
\newblock {\em J. Geom. Anal.}, 33(2):Paper No. 40, 34, 2023.

\bibitem{russo}
Aleksei~Vasil'evich Pogorelov.
\newblock On the stability of minimal surfaces.
\newblock In {\em Doklady Akademii Nauk}, volume 260, pages 293--295. Russian
  Academy of Sciences, 1981.

\bibitem{UY}
Masaaki Umehara and Kotaro Yamada.
\newblock Complete surfaces of constant mean curvature {$1$} in the hyperbolic
  {$3$}-space.
\newblock {\em Ann. of Math. (2)}, 137(3):611--638, 1993.

\bibitem{Yau}
Shing~Tung Yau.
\newblock Harmonic functions on complete {R}iemannian manifolds.
\newblock {\em Comm. Pure Appl. Math.}, 28:201--228, 1975.

\end{thebibliography}

\end{document}